\documentclass[12pt]{article}
\usepackage{amsmath, amssymb, amsthm,mathrsfs , fullpage}

\newtheorem{theorem}{Theorem}
\newtheorem{corollary}{Corollary}
\newtheorem{proposition}{Proposition}
\newtheorem{lemma}{Lemma}
\newtheorem{remark}{Remark}

\begin{document}

\title{Global existence and long term behavior of 2D electro-hydrodynamics}
\author{Rolf J. Ryham\\ {\it Department of Mathematics MS-136} \\ {\it Rice University}\\{\it Houston, TX 77006}} 
\maketitle

\abstract{We study the equations of a two dimensional  incompressible Newtonian fluid coupled with a dispersive parabolic-elliptic system
on bounded domains.
Global in time weak  solutions
are shown to exist  and converge with a rate to the stationary solution for $L^2$ initial data.
This paper extends and improves on a body of work surrounding the Debye-H\"uckel system to the 
hydrodyanamical case.}

\section{Introduction}
The equations of  electro-hydrodynamics are 
\begin{gather}
\label{NS}
\frac{\partial u}{\partial t} +  u \cdot  \nabla u + \nabla p = \Delta u + \Delta \phi \nabla \phi,\\
\label{DIV}
\nabla \cdot u  = 0,\\
\label{NP1}
\frac{\partial v}{\partial t} + u \cdot \nabla v  = \nabla \cdot \left(\nabla v -v \nabla \phi \right),\\
\label{NP2}
\frac{\partial w}{\partial t} + u \cdot \nabla w = \nabla \cdot \left( \nabla w +w \nabla \phi \right),\\
\label{poisson}
\Delta \phi = v - w.
\end{gather}
On bounded domains, solutions are determined by the conditions
\begin{gather}
\label{F1}
\frac{\partial v}{\partial \nu} -v \frac{\partial \phi}{\partial \nu} = 0,\\
\label{F2}
\frac{\partial w}{\partial \nu} +w \frac{\partial \phi}{\partial \nu} = 0,
\end{gather}
where $\nu$ is the outward pointing normal to $\partial U$ and 
\begin{equation}
\label{S}
u = 0  
\end{equation}
on $\partial U.$  In this paper solutions of  the Poisson equation \eqref{poisson} are determined by 
\begin{equation}
\label{D}
\phi = 0
\end{equation}
on $\partial U.$ 
The evolution is determined by initial conditions 
\begin{equation}
\label{I}u(x,0) = u_0(x), \quad v(x,0) = v_0(x), \quad w(x,0) = w_0(x)
\end{equation}
$v_0(x)$ and $w_0(x)$ are assumed to be nonnegative.  Positivity and the integral
of $v$ and $w$ are conserved throughout the dynamic. 

The equations of electro-hydrodynamics are the force balance equation
of an incompressible Newtonian fluid coupled with a parabolic system of conservation 
equations and an elliptic equation.  \eqref{NS} and \eqref{DIV} are the Navier-Stokes equations
with Lorentz force $\Delta \phi \nabla \phi.$ \eqref{NP1} and \eqref{NP2} are the conservation
of mass equations of charge densities.  $u$ is the velocity field of the fluid, $p$ is the pressure and $v$ and $w$ are 
densities proportional to the probability density of a system of negatively and positively respectively,
self repelling charged particles in the fluid and $\phi$ is the electrostatic potential due to Coulombic interactions
of the charged particles along with the charge exterior to the domain fixing the boundary condition of $\phi.$ The electrostatic stress exerted by
the charges on the fluid has the form of a rank one tensor and a pressure;
\begin{equation*}
\tau_{\mathrm{e}} = \nabla \phi \otimes \nabla \phi - \frac{1}{2}|\nabla \phi|^2 I.
\end{equation*}
$\tau_{\mathrm{e}}$ stems from the balance of kinetic energy
with electrostatic energy via the least action principle, \cite{RyLiZi07}.   For simplicity, we have
assumed that the density, viscosity, mobility and dielectric constant are unity.

The equations of electro-hydrodynamics are one of many fluid-particle systems
which have attracted much attention for there challenge with regard to mathematical existence theory, 
derivation and simulation.  Electro-hydrodnamic fluids are a particularly attractive
complex fluid due to there emerging application in microfluidic devices, electric biogels, 
switchable soft matter and proton exchange membranes, \cite{BeCh02, PaCa06, SqBa04}.
Some related works centered around other nonlinear Fokker-Plank and Navier Stokes 
systems and the equations of viscoelastic fluids can be found in \cite{Co05, Co07, CoFeTiZa07, CoMa08, LiLiZh05}.
In the case of electro-hydrodynamics, the closure of the nonlinear Fokker-Plank equations 
gives the Debye-H\"uckel system, a basic model  for the diffusion of ions in an electrolyte filling all of $\mathbb{R}^3$
first formulated by W. Nernst and M. Plank at the end of the nineteenth century, \cite{DeHu23}.

\subsubsection*{Main Result}
For the remainder of the paper, $U$ is assumed to a  connected, bounded  open subset of  $\mathbb{R}^2$ with class $C^{1,1}$ boundary $\partial U.$  
The space-time cylinder is $Q = U \times \mathbb{R}^+.$

The first result of this paper concerns the existence of 
global in time solutions.
\begin{theorem}  
\label{globalex}
Given $u_0 \in \mathsf{H}(U), v_0 \in L^2(U), w_0 \in L^2(U)$ there is a unique, global in time, weak solution $u,v,w$ of  \eqref{NS}-\eqref{D}.
\end{theorem}
The definition of a global in time, weak solution will be given in the proof of theorem \ref{globalex}
at the end of section \ref{WEAKSOLUTIONS}.  The proof of theorem \ref{globalex} is based on 
a modified Galerkin procedure found in \cite{LiLi95}. 
Formally setting $u \equiv 0,$ reduces  (\ref{NS}-\ref{I})  to
the Debye-H\"uckel system. \cite{BiHeNa94} have shown that the Debye-H\"uckel system has 
a unique, global in time, weak solution under the assumptions on $U$ above.
It is not known whether global in time, weak solutions of the Debye-H\"uckel system
exist on general smooth, bounded domains in dimensions greater than 2.

The second result of this paper concerns the long term behavior of the solutions guaranteed
by theorem \ref{globalex}.
\begin{theorem}
\label{l2ass}
There exist a positive constant $\lambda$ depending only on $U$ and constant $C_{\dagger}$ depending only on $u_0, v_0$
and $w_0$ such that 
\begin{equation}
\label{l2assin}
 \|u(t)\|_{\mathsf{H}}^2 + \|v(t) - V \|_{L^2(U)}^2 + \|w(t) - W\|_{L^2(U)}^2 +  \|\phi - \Phi\|_{H^1(U)}^2\leq C_{\dagger} e^{-\lambda t}.
\end{equation} 
for all $t \in \mathbb{R}^+$ where $V, W, \Phi$ is the unique steady-state solutions.
\end{theorem}
The definition of the steady state solution and stationary solution are given in section \ref{globalbehave}.
\cite{BiHeNa94} have shown that solutions of the Debye-H\"uckel system 
converge to the steady state solution in the $L^2$ norm in two and three dimensions provided the $L^2$ norm of solutions
is bounded in time.
Later, \cite{BiDo00} proved that solutions of the Debye-H\"uckel system   on uniformly convex domains in arbitrary dimensions
converge exponentially to the steady state solution,  provided the solutions  are defined globally in time.
With slight modifications, the proof of theorem \ref{l2ass} presented in this paper 
implies the exponential convergence of solutions of the Debye-H\"uckel system to the steady state solution in 
two and three dimensions without the assumption of uniform convexity on the domain.  See corollary \ref{nonuniform} 
at the end of  section \ref{globalbehave}.

\begin{remark}[Some generalizations]
In the forthcoming theory,  the Laplacian in \eqref{poisson} may be replaced by 
any operator $\mathscr{Q}$ associated with a convex quadratic form $Q(\cdot,\cdot)$ which 
is coercive over $H^1_0(U)$ and satisfies estimates \eqref{ellipin1} and \eqref{ellipin2} when $\Delta^{-1}$ 
is replaced by $\mathscr{Q}^{-1}.$
In general, the system may include several density functions $v_1, \dots, v_s$
with valences $\mu_1, \dots, \mu_s.$
Associated with such a  parabolic-elliptic system is the \emph{entropy}
\begin{equation*}
\mathscr{E}(v_1, \dots, v_s) = \int_{\Omega}
v_1 \log v_1 + \dots v_s \log v_s \,dx 
\end{equation*}
and the \emph{interaction energy}
\begin{equation*}
\mathscr{F}_Q(v_1, \dots, v_s) = 
\min_{\phi \in C^{\infty}_0(\Omega)}\left( Q(\phi, \phi) - \int_{\Omega} \phi (\mu_1 v_1 + \dots
+  \mu_s v_s)\,dx\right).
\end{equation*}
The \emph{kinetic energy} is 
\begin{equation*}
\mathscr{G}(u) = \int_U \frac{1}{2} |u|^2\,dx.
\end{equation*}
 The \emph{entropy functional} 
is\begin{equation*}
\mathscr{H}_Q(u, v_1, \dots, v_s)  = \mathscr{E}(v_1, \dots, v_s) - \mathscr{F}_Q(v_1, \dots, v_s) +  \mathscr{G}(u).
\end{equation*}
If $u$ is the motion of an incompressible, Newtonian fluid coupled with charge densities $v_1 \dots v_s$ 
with interaction energy $\mathscr{F}_Q$ then $\mathscr{H}_Q(u(t), v_1(t), \dots, v_s(t))$ is nonincreasing in $t.$
In this paper $s=2,$ $Q(v,v) = |\nabla v|^2/2,$ $\mu_1 = -1, v_1 = v, \mu_2 = 1, v_2 = w.$ 
\end{remark}

\subsection*{Notation}
The $L^p(U)$ norm will be denoted $\|\cdot \|_{L^p(U)}$ while the Sobolev norm of $W^{1,2}(U)$ will be denoted $\|\cdot\|_{H^1(U)}.$
The norm of a vector in Euclidean space will be denoted by $|\cdot|.$  We say a measurable function $f$  is \emph{nonnegative} if $f(x) \geq 0$ for a.e. $x \in U.$
$\|\cdot\|_{L \log L(U)}$ will denote the integral of $f\log f$ for a measurable, nonneagative function $f$
provided that integral exists.   
See Chapter 3 of \cite{TEMAM} for the defintion and properties of Banach space $X_0$ valued functions
$f$ in $L^p(E;X)$ with derivative $f'$ in $L^q(E; X')$ for $E$ an open subset of $\mathbb{R}$ 
and $1 \leq p,q \leq \infty.$
If $f' \in L^q(E; X')$  then we take $f$ to be the continuous representative
of its equivalence class.

In certain instances, it will be convenient to write the sum of the norm of two functions as follows;
\begin{equation*}
\|u,v\|^r_X = \|u\|^r_X + \|v\|^r_X
\end{equation*}
with $r = p$ when $X = L^p(U), H^1(U)$ or $L^p(E;X)$  and $r = 1$ when $X = L^1(U)$ or $L\log L(U).$

Let $\Delta^{-1}: L^2(U) \rightarrow H^2(U)\cap H^1_0(U)$ be defined by $\Delta^{-1} v = \phi$ provided $\Delta \phi = v$ and $\phi \in H^1_0(U).$

A note on constants: we will use $(\mathrm{const})$ to denote an inessential constant which may change from line to line.
The letter $C$ with various sub and superscripts will denote a constant refered to in various parts of the paper while 
constants  $c_1, c_2, c_3, c_{\Delta}$ or $c_{\Delta}'$ defined below will be written to  
indicate which inequality was used in that line. 

Constants $c_1, c_2$ and $c_3$ appear in the following 
versions of the Nash, Poincar\'e and Sobolev inequalities  
resp. in two dimensions (see \cite{LIEBLOSS});
\begin{equation*}
\begin{aligned}
\|v\|_{L^2(U)}^2 &\leq c_1  \|v\|_{H^1(U)}\|v\|_{L^1(U)}, \quad \forall v \in C^{\infty}(U),\\
\|v\|_{L^2(U)}^2 &\leq c_2  \| \nabla v\|_{L^2(U)}^2, \quad \forall v \in C^{\infty}_{\mathrm{c}}(U),\\
\|v\|_{L^3(U)} &\leq c_3 \|v\|^{1/2}_{H^1(U)}\|v\|^{1/2}_{L^2(U)}, \quad \forall v \in C^{\infty}(U).
\end{aligned}
\end{equation*}
Constants $c_{\Delta}$ and $c_{\Delta}'$ stemming  from regularity of solutions
of the Poisson equation on domains with $C^{1,1}$ boundary (see \cite{GILBARG}) will also be useful
\begin{equation}
\begin{aligned}
\label{ellipin1}
\|\Delta^{-1}v\|_{H^1(U)} &\leq c_{\Delta} \| v\|_{L^2(U)} ,\quad \forall v \in C_{\mathrm{c}}^{\infty}(U), \\
\|\nabla \Delta^{-1} v\|_{L^6} &\leq c_{\Delta}'  \| v\|_{L^3(U)}^{1/2}  \|\nabla \Delta^{-1} v\|_{L^2(U)}^{1/2}, \quad \forall v \in C_{\mathrm{c}}^{\infty}(U).
\end{aligned}\end{equation}
Finally, for $\epsilon > 0,$   \cite{BiHeNa94} have shown that there exists $C_{\epsilon} > 0$ depending only on $U$
\begin{equation}
\label{bilerin}
\|v\|_{L^3}^3 \leq \epsilon \|v\|_{H^1(U)}^2 \|v\|_{L\log L(U)} + C_{\epsilon} \|v\|_{L^1(U)},\quad  \forall v \in C^{\infty}(U). 
\end{equation}
provided $U$ is an open subset of $\mathbb{R}^2$ with $C^{1,1}$ boundary.

The space of smooth, compactly supported, divergence free vector fields is 
\begin{equation*}
\mathscr{V}(U) = \left\{ v \in (C^{\infty}_c(U))^2: \nabla \cdot v  = 0\right\}.
\end{equation*}
$\mathsf{H}(U)$ is the completion of $\mathscr{V}(U)$ is the completion of $\mathscr{V}(U)$ in the $L^2$-topology 
and $\mathsf{V}(U)$ is the completion of $\mathscr{V}(U)$ in the $H^1$ topology.
$\{U_i\}_{i=1}^{\infty}$ is the $L^2$-orthonormal basis of $\mathsf{H}(U)$ of the first component of eigenfunctions of
the Stokes operator with eigenvalues $\{\lambda_i\}_{i=1}^{\infty}.$ 
The dual space of $\mathsf{V}(U)$ is $\mathsf{V}'(U).$
See Chapter 1 of \cite{TEMAM} concerning these definition.

The constants $\mu_{v}, \mu_w, M_0 , R_0, R_0'$ and $S_0,$ which depend only on $u_0, v_0$ and $w_0$ 
will later be important;
\begin{equation*}
\begin{aligned}
\mu_v &= \int_U v_0\,dx, \\ \mu_w &= \int_U w_0 \,dx, \\
 M_0 &= \|u_0\|_{\mathsf{H}(U)}^2, \\
R_0 &= \max\left\{\|v_0\|_{L^2(U)}^2, \|w_0\|_{L^2(U)}^2 \right\}, \\
R'_0 &= \max \left\{\|v_0\|_{L \log L (U)},\|w_0\|_{L \log L (U)}\right\},\\
S_0 &= \|\nabla \Delta^{-1}(v_0 - w_0)\|_{L^2(U)}^2
\end{aligned}
\end{equation*}
are assumed to be finite.

\section{Weak Solutions}
\label{WEAKSOLUTIONS}
In this section $\nu$ is a positive integer and $t_0$ is a positive real number.

$B_1 \subset L^4([0,t_0]; L^4(U))$ is the ball of radius $R$ 
and $B_2 \subset C([0,t_0]; \mathbb{R}^{\nu})$ the ball of radius $M$
in their respective topologies.   

Define $\mathfrak{i}: C([0,t_0]; \mathbb{R}^{\nu}) \rightarrow C([0,t_0]; \mathscr{V}(U))$ 
by 
\begin{equation}
\mathfrak{i}(a) = \sum_{i=1}^{\nu} a_i U_i
\end{equation}
 and define $\mathfrak{j}: B_1 \times B_1 \rightarrow B_1\times B_1 \times B_1$ by
\begin{equation}
\mathfrak{j}(v,w) = (v,w, \nabla\Delta^{-1}(v-w)).
\end{equation}
Note that 
$\mathfrak{i}$ is an isometry from $C([0,t_0]; \mathbb{R}^{\nu})$ into $C([0,t_0]; \mathsf{H}(U)).$

Define two operators 
\begin{gather*}
\mathscr{X}: B_1 \times B_1 \times B_1 \rightarrow B_2\\
\mathscr{Y}: B_2 \rightarrow B_1 \times B_1 
\end{gather*}
as follows;  
$\mathscr{X}(v,w,e) = a$ provided $a$ is a solution to the $\nu$ dimensional system of 
ordinary differential equations
\begin{gather}
\label{ode}
a'_i + \lambda_i a_i + \sum_{j,k=1}^{\nu} a_j a_k \int_U U_j \cdot  \nabla U_k \cdot U_i \,dx  
= \int_U (v-w) e \cdot \nabla U_i\,dx\\
\label{odeinit}
a_i(0) = \int_U u_0 U_i\,dx, 
\mbox{ for all } i = 1,\dots, \nu \mbox{ and } t \in [0,t_0]
\end{gather}
and $\mathscr{Y}(a) = (v,w)$ provided $(v,w)$ is a weak solution of
\begin{gather*}
\frac{\partial v}{\partial t} +  (\mathfrak{i}(a))\cdot \nabla v   =\nabla \cdot \left(\nabla v - v \nabla \phi\right)\\
\frac{\partial w}{\partial t} + (\mathfrak{i}(a))\cdot \nabla w  = \nabla \cdot \left(\nabla w + w\nabla \phi \right)\\
\Delta \phi = v - w, \mbox{ for } t \in [0,t_0],\\
v(0,\cdot ) = v_0(\cdot), \quad w(0,\cdot) = w_0(\cdot).
\end{gather*}
Finally define 
\begin{equation*}
\mathscr{Z} = \mathscr{X} \circ \mathfrak{j} \circ \mathscr{Y}.
\end{equation*}

Let $C_*$ be the constant (which depends only on $U$) specified in lemma \ref{Llemma}
and $[0,t')$ be the interval of existence of the equation $f' = C_*f^3, f(0) = \max\{R_0,1\}.$
Let $C_0$ be the constant (which depends only on $R_0$ and $C_*$) specified 
in lemma \ref{Llemma}.  Let 
$R = C_0^{3/8}.$  Fix $0 < t^* < t'$ and let
\begin{equation*}
M^2 = (M_0 + R^4)e^{t^*}.
\end{equation*}
For these choices of $M, R$ and  $t^*,$ which depend only on the initial data and the domain, it will be shown that $\mathscr{Z}$ has 
a fixed point $a_{\nu}$ when $t_0 = t^*.$ 
Then, it will be shown  that there is a constant $M_1$ depending only on the initial data and the domain,
but independent of $M,R$ and $t_0$ such that any functions $a_{\nu}, v_{\nu}$ and $w_{\nu}$ corresponding to 
a fixed point of $\mathscr{Z}$ are bounded in the $l^2$ and $L^2$ norms resp.

N.B. in the proof of proposition \ref{extapprox} we will make a slight abuse of notation 
by assuming $u_0, v_0, w_0$ used in the definitions above are not necessarily the same 
as those functions given in the introduction, resulting in possibly different $M,R$ and $t_0.$


\begin{lemma}
$\mathscr{X}$ is well defined, continuous and $\mathscr{X}(B_1)$ is precompact in $C([0,t_0]; \mathbb{R}^{\nu}).$
\end{lemma}
\begin{proof}
\emph{(Well Definedness)}
Let $(v,w,e) \in B_1.$  The system \eqref{ode}, \eqref{odeinit} is a finite dimensional system of ordinary differential equations
with continuous dependence on $t$ and $a.$  By Peano's theorem, there exists
$\epsilon > 0$ so that $a$ solves \eqref{ode}, \eqref{odeinit} for $t \in [0,\epsilon].$
We no show that $a$ extends to $[0,t_0].$
 Multiplying \eqref{ode} by $a_i$ and summing over $i = 1,\dots, \nu$ we find that 
\begin{equation*}
\begin{aligned}
\sum_{i=1}^{\nu} \frac{1}{2}\frac{d}{dt} a_i^2 + \lambda_i a_i^2 &=  \sum_{i=1}^{\nu} a_i \int_U  (v-w) e\cdot  U_i \,dx\\
&\leq \left( \int_U |(v-w)e|^2 \,dx\right)^{1/2}\left( \int_U \big| \sum_{i=1}^{\nu} a_i U_i\big|^2 \,dx\right)^{1/2}\\
&\leq  \frac{1}{4}\int_U (v-w)^4 \,dx +  \frac{1}{4} \int_U |e|^4 \,dx + \frac{1}{2}\sum_{i=1}^{\nu} a_i^2.
\end{aligned}
\end{equation*} 
Letting $\omega(t) = \sum_{i=1}^{\nu} a_i^2(t)$ we see that $\omega$ satisfies
the differential inequality $\omega ' \leq f(t) + \omega$ with $\omega(0) \leq M_0$ and 
$\int_0^{t_0} f(s)\,ds \leq R^4.$  We infer that $\omega(t)$ is majorized by the function
$(M_0+R^4)e^t$ on the interval $[0,t_0]$ so that 
\begin{equation*}
\sum_{i=1}^{\nu} a_i^2(t) \leq (M_0+R^4)e^{t_0} = M^2, \mbox{ for all } t\in [0,t_0].
\end{equation*}
Hence $a \in B_2$ for all $t = [0,t_0].$

\emph{(Continuity)}
Since $\mathscr{X}$is well defined, $a, \bar a \in B_2$ provided $(v,w,e), (\bar v, \bar w, \bar e) \in B_1.$   Thus $|a_i(t)|, |\bar a_i(t)| \leq \sqrt{M}$ for all $t \in [0,t_0]$ and all $i = 1, \dots, \nu.$
Let 
\begin{equation*}
K_{\nu} = \max_{i,j,k = 1,\dots, \nu}  \max_{x \in \mathrm{Clos}(U)}  \left | U_i \cdot \nabla U_j \cdot U_k \right|.
\end{equation*}
Subtracting the equations solved by $a$ and $\bar a$ resp. from each other, multiplying by $a_i - \bar a_i$ and
summing over $i = 1,\dots, \nu$ we find that
\begin{equation*}
\begin{aligned}
&\sum_{i=1}^{\nu} \frac{1}{2} \frac{d}{dt}(a_i - \bar a_i)^2 + \lambda_i (a_i - \bar a_i)^2 \\
&= 
\int_U \sum_{i,j,k=1}^{\nu} \sum_{l,m=1}^2 (a_i a_j - \bar a_i \bar a_j) (a_k  - \bar a_k)   U_i \cdot \nabla U_j \cdot U_k \,dx \\
& + \int_U \sum_{i=1}^{\nu}  \left((v-w) e- (\bar v - \bar w)  \bar e \right)(a_i - \bar a_i)U_i \,dx\\
&\leq 2 \sqrt{\nu M}  K_{\nu} \sum_{i=1}^{\nu} (a_i - \bar a_i)^2 \\
& + \left( \int_U |(v-w) e  - (\bar v - \bar w)  \bar e|^2 \,dx\right)^{1/2}\left( \int_U \big| \sum_{i=1}^{\nu} (a_i - \bar a_i) U_i\big|^2 \,dx\right)^{1/2}\\
&\leq  (\mathrm{const}) \sum_{i=1}^{\nu} (a_i - \bar a_i)^2 + \frac{1}{2}\int_U |(v-w) e  - (\bar v - \bar w)  \bar e|^2 \,dx.
\end{aligned}
\end{equation*}
Letting $\eta(t) = \sum_{i=1}^{\nu} (a_i(t) - \bar a_i(t))^2$ we see that $\eta$ satisfies the differential inequality
$\eta' \leq (\mathrm{conts}) \eta + g(t)$ with $\eta(0) = 0.$   We have then for all $t \in [0,t_0]$  $|\eta (t)|\leq e^{(\mathrm{const} )t_0} \int_0^{t_0} |g(s)|\,dx$ 
where $\int_0^{t_0} |g(s)| \,ds $  converges to zero as $v,w,e$ converges
to $\bar v, \bar w,$ $\bar e$ in $ L^4([0,t_0]; L^4).$
 
\emph{(Compactness)}
 Fix $i  \in \{1,\dots,\nu\}$ and integrate \eqref{ode} over $[s,t] \subset [0,t_0]$ to find that
 \begin{equation*}
\begin{aligned}
&|a_i(s) - a_i(t)| \\
&= \left|  \int_s^t    \int_U \sum_{j,k=1}^{\nu}  a_j(r) a_k(r)    U_j  \cdot \nabla U_k \cdot U_i 
 + (v-w)e \cdot U_i \,dx -   \lambda_i a_i(r)  \,dr \right|\\
&\leq (\nu^2 M^2 K_{\nu} + \max_{i=1,\dots, \nu} \lambda_i M)(t-s)  +R^2 \sqrt{t-s} \leq K' \sqrt{t-s}
\end{aligned}
 \end{equation*}
 where $K ' =(\nu^2 M^2 K_{\nu} + \max_{i=1,\dots, \nu} \lambda_i M)  +R^2$ provided $|t-s| \leq 1.$
For $\epsilon > 0,$ choosing $\delta  = \min\{1, \epsilon^2/(\nu K')\}$ will show that $|a(s) - a(t)| < \epsilon$
for all $|t - s| < \delta$ and all $v,w,e$ in $B_1.$
Consequently the image of $B_1$ is uniformly equicontinuous in $B_2.$ The compactness asserted by the  lemma follows from
the  Arzela - Ascoli theorem.
 \end{proof}
 \begin{lemma} 
 \label{Llemma}
 $\mathscr{Y}$ is well defined and continuous.
 \end{lemma}
\begin{proof}
Let $a \in B_2$ and write $u = \mathfrak{i}(a).$
Let $N \in \mathbb{Z}^+$ to be chosen below and  $h = t_0/N.$  Define a sequence of functions in $H^1(U)$ as follows; 
for $i = \mathbb{Z}^+$ define $u^i = u|_{t = ih}$ let $v^i$ and $w^i$ solve 
\begin{gather*}
v^i = \left(I + h\Delta + hu^i \cdot \nabla  - h\nabla \cdot \left(v^{i}\nabla \phi^{i-1}\right)  \right)^{-1}
v^{i-1}\\
w^i = \left(I + h\Delta + hu^i \cdot \nabla + h\nabla \cdot \left(w^{i}\nabla \phi^{i-1}\right)\right)^{-1}
w^{i-1} \\
\Delta \phi_i = v_i - w_i
\end{gather*}
with and for $i = 0$ let $v^0 = v_0, w^0 = w_0.$
The sequence is well defined by the Lax-Milgram theorem for example, and in fact satisfies
\begin{equation*}
\begin{aligned}
\|v_i, w_i\|_2^2 &+ h\|Dv_i, Dw_i\|_2^2\\
 =& \int_U v^{i-1}v^i + w^{i-1} w^i\,dx  + h \int_U \nabla v^i \cdot \nabla \phi^{i-1} \nabla v^i 
- \nabla w^i \cdot \nabla \phi^{i-1} \nabla w^i  \,dx\\
=& I + II.
\end{aligned}
\end{equation*}
We have by H\"older's inequality
\begin{equation*}
I \leq \frac{1}{2}\|v^i,w^i\|_2^2 + \frac{1}{2}\|v^{i-1},w^{i-1}\|_2^2
\end{equation*}
and 
\begin{equation*}
\begin{aligned}
II &\leq \frac{h}{2}\|\nabla v^i,\nabla w^i\|_2^2 + \frac{h}{2}(\|v^i\|_3^2 +\|w^i\|_3^2)\|\nabla \phi^{i-1}\|_6^2\\
&\leq\frac{h}{2}\|\nabla v^i,\nabla w^i\|_2^2 + \frac{hc_{\Delta}c_3}{2}(\|v^i\|_{1,2}\|v^i\|_2 +\|w^i\|_{1,2}\|w^i\|_2)\|v_{i-1},w_{i-1}\|_2^2\\
&\leq \frac{3h}{4}\|\nabla v^i,\nabla w^i\|_2^2 + h(\mathrm{const})\|v^i,w^i\|_2^2(\|v_{i-1},w_{i-1}\|_2^4 + \|v_{i-1},w_{i-1}\|_2^2).
 \end{aligned}
\end{equation*}
If $\omega_i = \|v^i, w^i\|_2^2$ and $\mu_i =  \|\nabla v^i, \nabla w^i\|_2^2$ then
\begin{equation}
\label{polyin}
\omega_i + h\mu_i  \leq \omega_{i-1} + h(\mathrm{const})\omega_i(\omega_{i-1} + \omega_{i-1}^2).
\end{equation}
Let $C_*$ equal $2$ times
the constant in the previous line and $f$ solve $f' = C_*f^3$ with $f(0) = \max\{R_0,1\}.$
Then, for all $t \in [0,t_0],$ $f(t) \geq 1,$ $f'(t)$ is nondecreasing and  
\begin{equation*}
(\mathrm{const})f(t)(f(t) + f^2(t)) \leq C_*f^3(t) = f'(t) \leq \frac{1}{h}\int_t^{t+h} f'(s)\,ds = \frac{1}{h}(f(t+h) - f(t)).
\end{equation*}
Hence
 \begin{equation*}
 f(ih) + h(\mathrm{const})f((i+1)h)(f(ih) + f^2(ih)) \leq f((i+1)h),
 \end{equation*}
for all $ i = 0,\dots, N-1$ for $h$ sufficiently small and so the inequality
\begin{equation*}
\begin{aligned}
\omega_{i} &\leq \omega_{i-1} + h(\mathrm{const})\omega_i(\omega_{i-1} + \omega_{i-1}^2)\\
&\leq f(ih) + h(\mathrm{const})f((i+1)ih)(f(ih) + f^2(ih)) \leq f((i+1)h)
\end{aligned}
\end{equation*}
for all $i = 0, \dots, N$
follows by induction on $i$ with the case $i = 0$ being $\omega_0 \leq R_0 \leq f(0)$ for sufficiently small $h.$
Summing equations \eqref{polyin} for $i = 1, \dots N$ then implies
\begin{equation*}
\omega_N + \sum_{i=1}^N h\mu_i  
\leq  R_0 + (\mathrm{const})\int_0^{t_0} f(t)(f(t) + f(t)^2)\,dt
 \end{equation*}

  Define functions $v_h,w_h \in L^2([0,t_0];H^1(U))$ as follows;
\begin{equation*}
v_h(t) = v^i, w_h(t) = w^i, \mbox{ for } ih \leq t < (i+1)h \mbox{ and } i = 0,\dots, N-1.
\end{equation*}
Then
\begin{equation*}
\int_0^{t_0} \|v_h,w_h\|^2_{H^1}\,dt  = \sum_{i=0}^{N-1} h(\omega_i + \mu_i) \leq  C_0(t_0) := R_0 + (\mathrm{const})\int_0^{t_0} f(t)(f(t) + f(t)^2)\,dt .
\end{equation*}
The right-hand side of this inequality is bounded independently of $h.$
Consequently some subsequence $\{v_h,w_h\}_{h\downarrow 0}$ converge to $(v,w)$ in the weak topology
of $L^2([0,t_0];$ $H^1(U))$ and the weak-$*$ topology $L^{\infty}([0,t_0]; L^2(U)).$
$(u,v)$ satisfy
\begin{equation}
\label{weaksolbnd}
\|v,w\|_{L^4([0,t_0], L^4(U))} \leq \|v,w\|_{L^2([0,t_0];H^1(U))}^{1/2} \|v,w\|_{L^{\infty}([0,t_0];L^2(U))}^{1/4} \leq C_0^{3/8}.
\end{equation}
Let $\psi \in C^{\infty}([0,t_0]\times U),$
with $\mathrm{Supp}(\psi) \subset [0, t_0) \times U.$
Write $\psi^i(\cdot) = \psi(ih, \cdot)$   
and compute;
{\footnotesize
\begin{equation*}
\begin{aligned}
&\int_0^{t_0} \int_U \psi' v_h\,dx \,dt = \sum_{i=0}^{N-1} \int_{ih}^{(i+1)h}\int_U \psi' v^i \,dx \,dt = \sum_{i=0}^{N-1}  \int_U (\psi((i+1)h) - \psi(ih)) v^i \,dx\\
&= -\int_U \psi_0 v_0  \,dx - \sum_{i=1}^N \int_U \psi^i(v^i - v_{i-1}) \,dx \\
&= -\int_U \psi_0 v_0  \,dx  +
 \sum_{i=1}^N h \int_U  \left(\nabla v^i- 
 v^{i}\nabla \phi^i \right)\cdot \nabla \psi^i
-  u^i\cdot \nabla v^i \psi^i\,dx\\
 &= -\int_U \psi_0 v_0 \,dx  +
 \sum_{i=1}^N \int_{ih}^{(i+1)h} \int_U    \left(\nabla v^i- 
 v^{i}\nabla \phi^i \right)\cdot \nabla \psi^i
-  u^i\cdot \nabla v^i \psi^i\,dx\,dt.\end{aligned}
\end{equation*}
}The shift operator is continuous on $L^2([0,t_0]; H^1(U)).$ 
Thus choosing possibly another subsequence of $h\downarrow 0$ such that $\phi_h$ converges to $\phi$
in $L^2([0,t_0];H^1(U)),$ the equation converges to 
\begin{equation*}
\int_{Q_{t_0}} \psi' v \,dx \,dt = \int_U \psi(0,x) v_0\,dx + 
\int_{Q_{t_0}}\left(\nabla v - v \nabla \phi \right)\cdot \nabla \psi
-  u \cdot \nabla v  \psi
 \,dx.
\end{equation*}
The analogous equation holds for $w$ and $\Delta \phi = v -w$ for almost every $t \in [0,t_0].$

The well definedness (uniqueness of $v,w$) and the continuity of $\mathscr{Y}$ now follows from
standard estimates on the solutions $v,w$ (see for example \cite{BiHeNa94}.)
\end{proof}

\begin{proposition}
\label{fixedpoint}
$\mathscr{Z}$ has a unique fixed point $a_{\nu}.$  
\end{proposition}
\begin{proof} $t_0, M, R$ are defined so that $\mathscr{X}(\mathfrak{j}(B_1\times B_1))\subset B_2$ and $\mathscr{Y}(B_2) \subset B_1 \times B_1.$  
Thus $\mathscr{Z}(B_2) = \mathscr{X} \circ \mathfrak{j} \circ \mathscr{Y}(B_2) \subset B_2.$ Furthermore,
 $\mathscr{Z}$ is a composition
of two continuous  maps followed by a compact map and hence is compact.  A fixed point $a_{\nu}$ of
$\mathscr{Z}$ in $B_2$ is guaranteed by the Schauder fixed point theorem.  The uniqueness
of the fixed point is proved by standard estimates.
\end{proof}

\begin{proposition}
\label{fixptprop}
Let $0 < t_0 < \infty$ and suppose that  
 $(v_{\nu},w_{\nu}) = \mathscr{Y}(a_{\nu}),$
and $a_{\nu} = \mathscr{X}(\mathfrak{j}(v_{\nu}, w_{\nu}))$ exist for some choice of $M$ and $R.$
Then  there exists constant $M_1$ depending only on $M_0, R_0, R_0',$ $S_0, \mu_v, \mu_w$ and $U$,
but independent of $M,R$ and $t_0$ such that 
\begin{equation*}
\max_{t \in [0,t_0]}  \,|a_{\nu}(t)|^2 + \|v_{\nu},w_{\nu}\|_{L^2(U)}^2 \leq M_1.
\end{equation*}
\end{proposition}
\begin{proof}
Let $\phi_{\nu} = \Delta^{-1}(v_{\nu} - w_{\nu})$
and $u_{\nu} = \mathfrak{i}(a_{\nu}).$ Then one readily checks that 
\begin{equation}
\label{velen}
\begin{aligned}
&\sum_{i=1}^{\nu} \frac{1}{2} a_i^2(t) + \lambda_i \int_0^t a^2_i(s) \,ds \\
& =  
\sum_{i=1}^{\nu} \frac{1}{2} a_i^2(0)
  + \int_0^t \int_U (v_{\nu}(s) - w_{\nu}(s)) \nabla \phi_{\nu} \cdot u_{\nu}(s)\,dx \,ds
\end{aligned}
\end{equation}
for all $t \in [0,t_0].$

Let $f(t) = \|(v_{\nu})_-, (w_{\nu})_-\|_{L^2(U)}^2.$  $f$ is absolutely continuous with $f(0) = 0$ and
one readily checks that $f'(t) \leq 0$ for $t$ sufficiently small. Consequently $f(t) = 0$ for all $t \in [0,t_0]$
and $v_{\nu}(t)$ and $w_{\nu}(t)$ are nonnegative for all $t \in [0,t_0].$

For $h \geq 0$ let
\begin{equation*}
\begin{aligned}
\mathscr{K}_h(t) &= \int_U (v_{\nu}(t) + h) \log(v_{\nu}(t) + h) + (w_{\nu}(t) + h) \log(w_{\nu}(t) + h)\,dx\\
& - \int_U \left(\frac{1}{2}|\nabla \phi_{\nu}(t)|^2 + \phi_{\nu}(t)(v_{\nu}(t) - w_{\nu}(t))\right) \,dx.
\end{aligned}
\end{equation*}
$\mathscr{K}_h$ is similarly absolutely continuous and
\begin{equation*}
\begin{aligned}
\frac{d}{dt}\mathscr{K}_h(t) &= \int_U (\log(v_{\nu} + h)  + 1 - \phi_{\nu})v'_{\nu}  + (\log(w_{\nu} + h)  + 1 + \phi_{\nu})w'_{\nu}\,dx  \\
&+ \int_U  \nabla \phi_{\nu}\cdot \nabla \phi'_{\nu}   - \phi'_{\nu}(v_{\nu} - w_{\nu})\,dx
\end{aligned}
\end{equation*}
Note that the last three terms vanish for almost evert $t \in [0,t_0]$ because $\phi' = \Delta^{-1}(v'_{\nu} - w'_{\nu})$ for almost every $t\in [0,t_0].$  Continuing,
\begin{equation*}
\begin{aligned}
\frac{d}{dt}\mathscr{K}_h(t) &= \int_U  (\log(v_{\nu} + h)  + 1 - \phi_{\nu})v'_{\nu}  + (\log(w_{\nu} + h)  + 1 + \phi_{\nu})w'_{\nu} \,dx \\
&= -\int_U \nabla (\log(v_{\nu} + h)  + 1 - \phi_{\nu})\cdot \left(\nabla v_{\nu} - v_{\nu}\nabla \phi_{\nu}  +  v_{\nu} u_{\nu}\right) \,dx \\
&   -\int_U \nabla (\log(w_{\nu} + h)  + 1 + \phi_{\nu})\cdot \left(\nabla w_{\nu} + w_{\nu}\nabla \phi_{\nu}  +  w_{\nu} u_{\nu}\right) \,dx\\
&= -\int_U \frac{1}{v_{\nu} + h}\left(\nabla v_{\nu} - (v_{\nu} + h)\nabla \phi_{\nu}\right)\cdot \left(\nabla v_{\nu} - v_{\nu}\nabla \phi_{\nu} \right) \,dx \\
& -\int_U \frac{1}{w_{\nu} + h}\left(\nabla w_{\nu} + (w_{\nu} + h)\nabla \phi_{\nu} \right)\left(\nabla w_{\nu} + w_{\nu}\nabla \phi_{\nu} \right) \,dx \\
&- \int_U(v_{\nu} - w_{\nu})\nabla \phi_{\nu} \cdot  u_{\nu}\,dx.
\end{aligned}
\end{equation*}
Integrating this expression over $[0,t]$ for $t \leq t_0$ and adding it to \eqref{velen} (the last terms in both expressions cancel) gives
\begin{equation*}
\begin{aligned}
\sum_{i=1}^{\nu} \frac{1}{2} a_i^2(t) + \mathscr{K}_h(t) &\leq \sum_{i=1}^{\nu} \frac{1}{2} a_i^2(0) + \mathscr{K}_h(0) -
\int_0^t \int_U \frac{1}{v_{\nu} + h}\left|\nabla  v_{\nu} - v_{\nu}\nabla \phi_{\nu}\right|^2 \,dx \,ds\\
& - \int_0^t \int_U \frac{1}{w_{\nu} + h}\left|\nabla  w_{\nu} + w_{\nu}\nabla \phi_{\nu}\right|^2 \,dx \,ds + \eta(t)
\end{aligned}
\end{equation*}
where 
\begin{equation*}
\begin{aligned}
\eta(t) &= \int_0^t \int_U  \frac{h}{v_{\nu} + h}\nabla \phi_{\nu}\cdot \left(\nabla v_{\nu} - v_{\nu}\nabla \phi_{\nu} \right)\,dx \,dts \\ 
&- \int_0^t \int_U \frac{h}{w_{\nu} + h}\nabla \phi_{\nu}\left(\nabla w_{\nu} + w_{\nu}\nabla  \phi_{\nu}\right) \,dx \,ds
\end{aligned}
\end{equation*}
is majorized by the two terms appearing directly to the left of it and 
\begin{equation*}
\frac{\sqrt{h}}{2}\int_0^t \int_U |\nabla \phi|^2\,dx\,ds.
\end{equation*}
By bounded convergence we have
\begin{equation}
\label{totin}
\begin{aligned}
\sum_{i=1}^{\nu} \frac{1}{2} a_i^2(t) + \mathscr{K}_0(t) & \leq  \sum_{i=1}^{\nu} \frac{1}{2} a_i^2(t) + \lim_{h\downarrow 0} \mathscr{K}_h(t)\\
&  \leq \sum_{i=1}^{\nu} \frac{1}{2} a_i^2(0) + \lim_{h\downarrow 0}\mathscr{K}_h(0)  +  \frac{\sqrt{h}}{2}\int_0^t \int_U |\nabla \phi|^2\,dx\,ds\\
& = \sum_{i=1}^{\nu} \frac{1}{2} a_i^2(0) + \mathscr{K}_0(0)
\end{aligned}
\end{equation}
for all $t \in [0,t_0].$

We infer from \eqref{totin} that 
\begin{equation*}
|a_{\nu}(t)|^2 + \|v_{\nu}(t), w_{\nu}(t)\|_{L\log L(U)} + \|\nabla \phi_{\nu}(t)\|_{L^2(U)}^2\leq C' := M_0 + 2R_0' + 2S_0
\end{equation*}
for all $t \in [0,t_0].$ We use this result to make an energy estimate of $v_{\nu}$ and $w_{\nu}.$
We estimate 
\begin{equation*}
\begin{aligned}
&\int_{U} v_{\nu} \nabla \phi_{\nu}\cdot \nabla  v_{\nu} \,dx 
\leq
\|\nabla v_{\nu}\|_{L^2(U)}\|v_{\nu}\|_{L^3(U)} \|\nabla \phi_{\nu}\|_{L^6(U)}\\
& \leq c_{\Delta}' \|\nabla v_{\nu}\|_{L^2(U)}\|v_{\nu},w_{\nu}\|_{L^3(U)}\|v_{\nu} -w_{\nu}\|_{L^3}^{1/2}\|\nabla \phi\|_{L^2(U)}^{1/2}\\
&\leq c_{\Delta}' \sqrt{C'}\|\nabla v_{\nu}\|_{L^2(U)}\left( \epsilon \|v_{\nu},w_{\nu}\|_{H^1(U)}^2 \|v_{\nu},w_{\nu}\|_{L\log L(U)} + C_{\epsilon} \|v_{\nu},w_{\nu}\|_{L^1(U)} \right)^{1/2}\\
&\leq c_{\Delta}' \sqrt{C'}\|\nabla v_{\nu}\|_{L^2(U)}\left(C'\epsilon \|v_{\nu},w_{\nu}\|_{H^1(U)}^2  + C_{\epsilon} (\mu_v + \mu_w) \right)^{1/2}\\
&\leq 3c_{\Delta}' C' \sqrt{\epsilon} \|v_{\nu},w_{\nu}\|_{H^1(U)}^2  + \frac{1}{2\epsilon}\sqrt{C'}C_{\epsilon} (\mu_v + \mu_w) 
\end{aligned}
\end{equation*}
An analogous inequality holds for $w_{\nu}.$
Choose $\epsilon$ sufficiently small so that $3c_{\Delta}' C' \sqrt{\epsilon} < 1/4$ so that
\begin{equation*}
\begin{aligned}
&\frac{d}{dt}\|v_{\nu}, w_{\nu}\|_{L^2(U)}^2 +\|v_{\nu}, w_{\nu}\|_{L^2(U)}^2 + 2\|\nabla v_{\nu}, \nabla w_{\nu}\|_{L^2(U)}^2\\
&\leq 2\|v_{\nu}, w_{\nu}\|_{L^2(U)}^2 +  \|\nabla v_{\nu},\nabla w_{\nu}\|_{L^2(U)}^2  + \frac{2}{\epsilon}\sqrt{C'}C_{\epsilon} (\mu_v + \mu_w) \\
&\leq 2c_1 \|\nabla v_{\nu},\nabla w_{\nu}\|_{L^2(U)}\|v_{\nu}, w_{\nu}\|_{L^1(U)} +  \|\nabla v_{\nu},\nabla w_{\nu}\|_{L^2(U)}^2  + \frac{2}{\epsilon}\sqrt{C'}C_{\epsilon} (\mu_v + \mu_w) \\
&\leq \frac{3}{2}\|\nabla v_{\nu},\nabla w_{\nu}\|_{L^2(U)}^2 + \frac{2}{\epsilon}\sqrt{C'}C_{\epsilon}(\mu_v + \mu_w)   + 2c_1^2(\mu_v + \mu_w)^2
\end{aligned}
\end{equation*}
Let 
$M_1' = 2(\epsilon)^{-1}\sqrt{C'}C_{\epsilon} (\mu_v + \mu_w)  + 2c_1^2(\mu_v + \mu_w)^2 + 2R_0 .$
Then $\omega(t) = \|v_{\nu}(t), w_{\nu}(t)\|_{L^2(U)}^2$ and $\eta =$$ \|\nabla v_{\nu}(t), \nabla w_{\nu}(t)\|_{L^2(U)}^2/2$
satisfy
\begin{equation*}
\frac{d}{dt} \omega + \omega + \eta \leq  M_1, \quad \omega(0) < M_1
\end{equation*}
so that $w(t) < M_1$ for all $t\in [0,t_0].$  
Setting
$M_1 = M_1' + C'
$ the theorem is proved.
\end{proof}
\begin{proposition}
\label{extapprox}
There exist functions 
\begin{gather*}
a_{\nu} \in C(\mathbb{R}^+; \mathbb{R}^{\nu})\\
 v_{\nu},  w_{\nu} \in C(\mathbb{R}^+; L^2)\ \cap L^2_{\mathrm{loc}}(\mathbb{R}^+; H^1(U))
\end{gather*}
such that for any $0 < t_0 < \infty,$  $( v_{\nu}, w_{\nu}) = \mathscr{Y}( a_{\nu}),$
and $ a_{\nu} = \mathscr{X}(\mathfrak{j}( v_{\nu},  w_{\nu}))$ and 
\begin{equation*}
\max_{t \in \mathbb{R}^+}  \left( | a_{\nu}(t)|^2 + \| v_{\nu}(t), w_{\nu}(t)\|_{L^2(U)}^2 \right)< M_1.
\end{equation*}
\end{proposition}
\begin{proof}
Let $T$ be the maximal positive real number for which the theorem holds.  Clearly $t_0 \geq t^*  > 0.$
Then $| a_{\nu}(t_0)|^2, \| v_{\nu}(t_0),  w_{\nu}(t_0)\|_{L^2(U)}^2  < \infty.$
Consider the functions $ a^*_{\nu},$ $( v^*_{\nu},  w^*_{\nu}) = \mathscr{Y}( a^*_{\nu})$ guaranteed by 
Theorem \ref{fixedpoint} 
when we introduce $u^*_0 =  \mathfrak{i}(a^*_{\nu}(t_0)),$ $ v^*_0 =  v_{\nu}(t_0),$ and $  w^*_0 =  w_{\nu}(t_0)$
(also defining $ M^*_0,R^*_0, R^{*'}_0$ and $S^*_0$ which are finite in terms of $u^*_0, v^*_0$ and $ w^*_0$.)   
Let $\tilde a_{\nu}(t) = a_{\nu}(t)$ for $t \leq T$ and $\tilde a_{\nu}(t) = a^*_{\nu}(t - t_0)$ for $t > T.$ 
Define $\tilde v_{\nu}$ and $\tilde w_{\nu}$ in the same way.
Then there exists $t^* > 0$ with $0 < \tilde t_0 = T + t^* < \infty$ and some $\tilde M$ and $\tilde R$ 
for which $(\tilde v_{\nu} , \tilde w_{\nu} ) = \mathscr{Y}(\tilde a_{\nu}),$
$\tilde a_{\nu}= \mathscr{X}(\mathfrak{j}(\tilde v_{\nu}, \tilde w_{\nu} )).$
By Theorem \ref{fixptprop},
\begin{equation*}
\max_{t \in [0,\tilde t_0]}  \,| \tilde a_{\nu} (t)|^2 + \| \tilde v_{\nu}(t), \tilde w_{\nu}(t)\|_{L^2(U)}^2 < M_1.
\end{equation*}
Then $\tilde t_0 > T$  contradicts the maximality of $T$ and concludes the proof.
\end{proof}
With the help of proposition \ref{extapprox} we are now able to supply the 
\begin{proof}[Proof of Theorem \ref{globalex}]
$(u,v,w)$ is said to be a global in time, weak solution of (\ref{NS} - \ref{I}) provided
\begin{gather*}
u \in L^2(\mathbb{R}^+; \mathsf{V}(U)) \cap C(\mathbb{R}^+; \mathsf{V}(U)) \mbox{ with } u' \in L^2(\mathbb{R}^+; \mathsf{V}'(U))\\
v,w \in L^2_{\mathrm{loc}}(\mathbb{R}^+; H^1(U)) \cap C(\mathbb{R}^+; L^2(U)) \mbox{ with } v',w' \in L^2_{\mathrm{loc}}(\mathbb{R}^+; H^{-1}(U))
\end{gather*}
and 
if for all  $t \in \mathbb{R}^+$ and all summable test functions $f \in C^{\infty}(\mathbb{R}_+; \mathscr{V}(U)),$
$g,h,\psi \in C^{\infty}(\mathrm{Clos}(Q))$ with $\psi|\partial U = 0,$
\begin{equation*}
\begin{aligned}
 \int_{\mathbb{R}^+} \langle u', f\rangle \,dt+ \int_Q   u \cdot \nabla u \cdot f \,dx \,dt &=
-  \int_Q\nabla u \cdot \nabla f - (v-w)\nabla \phi \cdot f \,dx \,dt\\
 \int_{\mathbb{R}^+}\langle v', g\rangle\,dt + \int_Q u \cdot \nabla  v g\,dx \,dt &=
-\int_Q
 \left(\nabla v - v\nabla \phi\right)\cdot \nabla g\,dx \,dt\\
 \int_{\mathbb{R}^+} \langle w' ,h\rangle \,dt + \int_Q u \cdot \nabla w  h\,dx \,dt 
&=
- \int_Q\left(\nabla w + w\nabla \phi \right)\cdot \nabla h \,dx \,dt\\
 \int_Q \nabla \phi \cdot \nabla  \psi \,dx \,dt 
 &= -\int_Q (v - w)\psi \,dx \,dt
\end{aligned}
\end{equation*}
Let $a_{\nu}, v_{\nu}$ and $w_{\nu}$ be as in proposition \ref{extapprox}.
Let  $T \geq 0.$ 
Let $u_{\nu} = \mathfrak{i}(a).$
Then there is a constant $C_T < \infty$ such that
\begin{equation}
\label{gbound}
\begin{aligned}
&\max_{t \in \mathbb{R}^+}  \,| u_{\nu}(t)|^2_{\mathsf{H}(U)} + \| v_{\nu}(t), w_{\nu}(t)\|_{L^2(U)}^2
+\int_0^T \| u_{\nu}(t)\|^2_{\mathsf{V}(U)} +\| u'_{\nu}(t)\|^2_{\mathsf{V}'(U)}\,dt\\
&+ \int_0^T
\| v_{\nu}(t), w_{\nu}(t)\|_{H^{1}(U)}^2 +
\| v'_{\nu}(t), w'_{\nu}(t)\|_{H^{-1}(U)}^2 \,dt < C_T.
\end{aligned}
\end{equation}
We have then
for some subsequence reindexed by $\nu$ 
\begin{gather}
\label{l2weaku}
u_{\nu} \rightharpoonup u \mbox{ in } L^2([0,T]; \mathsf{V}(U)), 
\quad
u_{\nu} \rightharpoonup_* u \mbox{ in } L^{\infty}([0,T]; \mathsf{H}(U))\\
u'_{\nu} \rightharpoonup u' \mbox{ in } L^2([0,T]; \mathsf{V}'(U))\\
\label{l2weakvw}
v_{\nu},w_{\nu} \rightharpoonup v, w \mbox{ in } L^2([0,T]; H^1(U)), 
\quad
v_{\nu},w_{\nu} \rightharpoonup_* v, w \mbox{ in } L^{\infty}([0,T]; L^2(U))\\
v'_{\nu},w'_{\nu} \rightharpoonup v', w' \mbox{ in } L^2([0,T]; H^{-1}(U)), 
\end{gather} 
for some functions $u,v$ and $w.$   We will see later that $u,v$ and $w$
do not depend on the $T.$  By the Aubin-Lion compactness lemma (Theorem 2.3 in \cite{TEMAM}),
we have
\begin{equation}
\label{l2conv}
u_{\nu} \rightarrow u \mbox{ in } L^2([0,T]; \mathsf{H}(U)), \quad
v_{\nu},w_{\nu} \rightarrow v,w \mbox{ in } L^2([0,T]; L^2(U)).
\end{equation}
It follows from elliptic regularity that there is also $\phi$ with (after choosing possibly another
subsequence)
\begin{equation}
\label{h2conv}
\phi_{\nu} \rightarrow \phi \mbox{ in } L^2([0,T]; H^2(U)).
\end{equation}

Let $g \in C^{\infty}(\mathrm{Clos}(Q))$
with  $\mathrm{Supp}(g) \subset [0,T) \times U.$ 
By the triangular ineqaulity, the equality $\mathscr{Y}(a_{\nu}) = (v_{\nu}, w_{\nu})$ implies 
\begin{equation*}
\begin{aligned}
&\left| \int_{\mathbb{R}^+}\langle v', g\rangle\,dt + \int_Q u \cdot \nabla  v g 
+\int_Q
 \left(\nabla v - v\nabla \phi\right)\cdot \nabla g\,dx \,dt\right| =\\
 &\lim_{\nu \rightarrow \infty} 
\left|
\int_{\mathbb{R}+} \langle v_{\nu}', g\rangle\,dt + \int_Q u_{\nu} \cdot \nabla  v_{\nu} g 
+\int_Q
 \left(\nabla v_{\nu} - v_{\nu} \nabla \phi_{\nu}\right)\cdot \nabla g\,dx \,dt
\right.\\
 &-\left.
\int_{\mathbb{R}^+}\langle v', g\rangle\,dt + \int_Q u \cdot \nabla  v g 
+\int_Q
 \left(\nabla v - v\nabla \phi\right)\cdot \nabla g\,dx \,dt\right| = 0
\end{aligned}
\end{equation*}
 provided we can demonstrate the convergence of the individual terms.  The convergence of the linear terms
follows from the definition of weak convergence.  We check convergence of the quadratic terms;
\begin{equation*}
\begin{aligned}
&\lim_{\nu \rightarrow \infty} \left| \int_Q\left( u_{\nu}\cdot \nabla  v_{\nu} - u \cdot \nabla v\right)g\,dx \,dt \right|\\
& \leq \lim_{\nu \rightarrow \infty}\|u_{\nu} - u\|_{L^2([0,T]; \mathsf{H}(U))}\|\nabla v_{\nu}\|_{L^2([0,T]; H^1(U))}\sup_Q |g| \\
&+\lim_{\nu \rightarrow \infty} \left| \int_Q u \cdot \nabla(  v_{\nu} -  v )g\,dx \,dt \right| = 0.
\end{aligned}
\end{equation*}
where the first term converges by \eqref{gbound}, \eqref{l2weakvw}    and \eqref{l2conv}.
 Similarly 
\begin{equation*}
\begin{aligned}
&\lim_{\nu \rightarrow \infty} \left| \int_Q\left( v_{\nu} \nabla \phi_{\nu}  - v \nabla \phi\right)\cdot \nabla g\,dx \,dt \right|\\
& \leq \lim_{\nu \rightarrow \infty} \|v_{\nu} - v\|_{L^2([0,T]; L^2(U))}\|\nabla \phi_{\nu}\|_{L^2([0,T]; H^1(U))}\sup_Q |\nabla g|\\
&+\lim_{\nu \rightarrow \infty}\|v\|_{L^2([0,T]; L^2(U))}\|\nabla (\phi_{\nu} - \phi)\|_{L^2([0,T]; H^1(U))}\sup_Q |\nabla g| = 0.
\end{aligned}
\end{equation*}
The analogous limits hold for $w_{\nu}$ and $h \in C^{\infty}(\mathrm{Clos}(Q)).$  
Let $\psi \in C^{\infty}(\mathrm{Clos}(Q))$ with $\psi|\partial U = 0.$
By  \eqref{l2conv} and \eqref{h2conv},
\begin{equation*}
\lim_{\nu \rightarrow \infty} \int_Q \nabla  \phi_{\nu}\cdot \nabla  \psi \,dx \,dt 
 = -\int_Q (v - w)\psi \,dx \,dt.
\end{equation*}

Let $F$ be of the form $F(t,x)= \sum_{i=j}^{\nu} b_j(t) U(x)$ for $\{b_j(\cdot)\}_{j=1}^{\nu}$
continuous. By the triangular ineqaulity, the equality $\mathscr{X}(\mathfrak{j} (v_{\nu}, w_{\nu})) = a_{\nu}$ implies 
\begin{equation*}
\begin{aligned}
&\left|
\int_{\mathbb{R}^+} \langle u',   F\rangle \,dt+ \int_Q   u \cdot \nabla u \cdot F 
+  \int_Q \nabla u \cdot \nabla F  - (v-w)\nabla \phi \cdot F \,dx \,dt  \right| =\\
&\lim_{\nu \rightarrow \infty} 
\left|
\int_{\mathbb{R}^+} \langle u'_{\nu},   F\rangle \,dt+ \int_Q   u_{\nu} \cdot \nabla u_{\nu} \cdot F 
+  \int_Q \nabla u_{\nu} \cdot \nabla F  - (v_{\nu}-w_{\nu})\nabla \phi_{\nu} \cdot F \,dx \,dt
\right.\\
 &-\left.
\int_{\mathbb{R}^+} \langle u',   F\rangle \,dt+ \int_Q   u \cdot \nabla u \cdot F 
+  \int_Q \nabla u \cdot \nabla F  - (v-w)\nabla \phi \cdot F \,dx \,dt  \right| = 0
\end{aligned}
\end{equation*}
provided we can demonstrate the convergence of the individual terms.  The convergence of the linear terms
follows from the definition of weak convergence.  Convergence of the first quadratic term
can be found in \cite{TEMAM}.  We check the last term
\begin{equation*}
\begin{aligned}
&\lim_{\nu \rightarrow \infty} \left| \int_Q\left( (v_{\nu}-w_{\nu})\nabla \phi_{\nu} - (v-w)\nabla \phi\right) \cdot F\,dx \,dt \right|\\
& \leq \lim_{\nu \rightarrow \infty}\|v_{\nu} - v, w_{\nu} - w\|_{L^2([0,T]; L^2(U))}\|\nabla \phi_{\nu}\|_{L^2([0,T]; L^2(U))}\sup_Q |F| \\
&+\lim_{\nu \rightarrow \infty} \|v-w\|_{L^2([0,T]; L^2(U))}\|\nabla (\phi_{\nu}  - \phi)\|_{L^2([0,T]; L^2(U))}\sup_Q |F| = 0.
\end{aligned}\end{equation*}
by \eqref{l2conv} and \eqref{h2conv}.

Finally,  
\begin{equation*}
\int_{\mathbb{R}^+} \langle u',  f\rangle  \,dt+ \int_Q   u \cdot  \nabla u \cdot f 
+ \nabla u \cdot \nabla f  - (v-w)\nabla \phi \cdot f  \,dx \,dt = 0
\end{equation*}
for $f  \in C^{\infty}([0,T]; \mathscr{V})$  by approximating $f$ by functions of the form 
$\sum_{i=j}^{\nu} b_j(t) U_i$ in the $C^{k}$ topology for $k = 1,2, \dots$

The uniqueness of $u,v,w$, which implies the extension of $u,v,w$ to $t \in \mathbb{R}^+,$ and the fact that 
\begin{equation*}
\lim_{t\rightarrow 0} u(t) = u_0, \quad
\lim_{t\rightarrow 0} v(t) = v_0, \quad
\lim_{t\rightarrow 0} w(t) = w_0, 
\end{equation*}
now follow from standard estimates.
\end{proof}

\section{Long Term Behavior Behavior}
\label{globalbehave}

Define a functional $\mathscr{J}$ on $H^1_0(U)$ and two absolutely continuous functions $\mathscr{K}$ and $\mathscr{L}$ over 
$\mathbb{R}_+$ in terms of functions $u,v,w,\phi;$
\begin{gather}
\mathscr{J}(\phi) = \int_U \frac{1}{2}|\nabla \phi|^2 \,dx + \mu_v \log \left( \int_U \exp \phi \,dx \right) + \mu_w \log \left( \int_U ( \exp \phi)^{-1} \,dx \right)\\
\mathscr{K}(t) = \int_U v(t)\log v(t) + w(t)\log w(t) + \frac{1}{2}|\nabla \phi(t)|^2 + \frac{1}{2}|u(t)|^2\,dx\\
\mathscr{L}(t) = \int_U \frac{\Theta}{2}|u|^2 + \frac{1}{2}\frac{(v(t) - V)^2}{V} +\frac{1}{2}\frac{(w(t) - W)^2}{W} + |\nabla (\phi(t) - \Phi)|^2 \,dx
\end{gather}
$\mathscr{K}$ is the entropy function of electro-hydrodynamics while $\mathscr{L}$ is stems form a  function introduced in \cite{AbMeVa04}
to study convergence of the Debye system to the steady state solution assuming Dirichelet boundary conditions.
$\Theta$ is a positive constant which will be specified in lemma \ref{Ldiff}.
$\mathscr{J}(\cdot)$ is strictly convex and bounded from below.  
There exists a unique function $\Phi \in C^{\infty}(U) \cap C_0(U)$ such that $\mathscr{J}(\Phi) < \mathscr{J}(\phi)$ 
for all $\Phi \neq \phi \in H^1_0(U),$  c.g. \cite{GoLi89};
Define functions
\begin{equation*}
{V}(x) =  \mu_v  \frac{\exp \Phi(x)}{\int_U \exp \Phi(x) \,dx},
\quad {W}(x) =  \mu_w  \frac{\exp (-\Phi(x))}{\int_U \exp (-\Phi(x)) \,dx }, \quad \forall x\in U.
\end{equation*}
We call ${V}, {W}, \Phi$ the steady state solution.  We will frequently us the fact that 
there are constants $a,b, a', b'$ for which
\begin{equation*}
0 < a \leq V(x) \leq b < \infty,  \quad 0 < a' \leq W(x) \leq b' < \infty, \quad \forall x \in U.
\end{equation*}
We call $V, W ,\Phi , U$  the stationary solution when $U \equiv 0.$  An important 
fact about $\Phi$ is that 
\begin{equation}
\label{pressure}
\Delta \Phi \nabla \Phi = \nabla (V + W)
\end{equation}
 is the gradient of a pressure
so that the stationary equations are consistent.

We first recall a well known fact about entropy functions which holds additionally with the kinetic energy term 
$|u|^2/2$ found in $\mathscr{K}.$ 
The manipulations in differentiating $\mathscr{K}$ may be justified by approximating $v,w$ by strictly 
positive functions or by the argument used in the proof of proposition  \ref{fixptprop}.
\begin{lemma}
\label{Kdecay}
\begin{equation*}
\begin{aligned}
\mathscr{K}(t) - \mathscr{K}(0)&= - \int_0^t \int_U |2\nabla \sqrt{v(s)} - \sqrt{v(s)} \nabla \phi(s)|^2\,dx \,ds \\
&- \int_0^t \int_U|2\nabla \sqrt{w(s)} + \sqrt{w(s)} \nabla \phi(s)|^2 + |\nabla u(s)|^2\,dx \,ds  \leq 0.
\end{aligned}
\end{equation*}
\end{lemma}
\begin{proof}
The proof is contained in the proof of proposition \ref{fixptprop}.
\end{proof}

\begin{lemma}[Weighted Poincar\'e Inequality]
\label{wpi}
Let $\Omega$ be an open subset of $\mathbb{R}^d$ and $\rho \in H^1(\Omega)$ satisfy
\begin{equation*}
0 <  \rho(x) \leq b < \infty, \quad \forall x \in \Omega
\end{equation*}
with $\rho^{-1}$ integrable.
Then there exists a $C_{\rho} = C(\rho, \Omega)$ such that 
\begin{equation*}
\int_\Omega f^2 \,dx \leq C_{\omega} \int_\Omega \left| \nabla (f \rho)  \right|^2\,dx
\end{equation*}
whenever
\begin{equation*}
\int_\Omega f \,dx = 0.
\end{equation*}
\end{lemma}
\begin{proof}
Suppose that no such constant exists.  Then there is a sequence of functions $\{f_i\}_{i=1}^{\infty}$ in $H^1(\Omega)$
with
\begin{equation*}
\int_\Omega f_i \,dx = 0,\quad
\|f_i\|_{L^2(\Omega)}^2 \geq i \int_{\Omega} |\nabla(f_i \rho)|^2 \,dx.
\end{equation*}
Let $h_i = f_i/\|f_i\|_{L^2(\Omega)}$ so that 
\begin{equation}
\label{normalize}
1 = \|h_i\|_{L^2(\Omega)}^2 \geq i \int_{\Omega} |\nabla(h_i \rho)|^2 \,dx.
\end{equation}
Write $g_i  = h_i\rho.$  Then 
\begin{equation*}
\int_{\Omega} g_i^2 \,dx = \int_{\Omega} h_i^2 \rho^2 \,dx \leq b^2 
\end{equation*}
shows that $g_i$ is bounded in $H^1(\Omega)$ and thus converges weakly to an element 
$g \in H^1(\Omega).$   Fatou's lemma
\begin{equation*}
\int_{\Omega} |\nabla g|^2 \,dx \leq \liminf_{i \rightarrow \infty} \int_{\Omega} |\nabla g_i|^2 \,dx \leq \lim_{i \rightarrow \infty} \frac{1}{i} = 0
\end{equation*}
shows that $g(x) = G$ a.e. for some constant $G.$     Then
{\footnotesize
\begin{equation*}
G\int_{\Omega} \rho^{-1} \,dx = \int_{\Omega} \rho^{-1} g\,dx = \lim_{i\rightarrow 0} \int_{\Omega}  \rho^{-1} g_i \,dx 
= \lim_{i\rightarrow 0} \int_{\Omega}  h_i \,dx= \lim_{i\rightarrow 0} \frac{1}{\|f_i\|_{L^2(\Omega)}}\int_{\Omega}  f_i \,dx = 0.
\end{equation*}
}We infer $G = 0$ since $\rho^{-1}$ is nonzero on a set of positive measure.  The contradiction with \eqref{normalize} completes the proof.
\end{proof}

\begin{lemma}
Define functions $g, h \geq 0$ via the formula  
\begin{equation*}
g(x,t) V(x,t) = v(x,t), \quad h(x,t) W(x,t) = w(x,t)
\end{equation*}
for $(x,t) \in U \times [0,\infty).$
Then
\begin{gather}
\label{conserve}
 \int_U g V \,dx = \mu_v,\quad   \int_U h W \,dx = \mu_w, \quad \forall t \in [0,\infty)
\end{gather}
and $g$ and $h$ are generalized solutions of the equations 
\begin{gather}
\label{gsolve}
 \frac{\partial (gV)}{\partial t} + u \cdot \nabla  (gV) = \nabla \cdot  \left(V\left(\nabla g  - g\nabla \Psi \right)\right),
\quad
\frac{\partial g}{\partial \nu}  - g\frac{\partial \Psi}{\partial \nu} = 0,\\
\label{hsolve}
 \frac{\partial (hW)}{\partial t} +  u \cdot \nabla (hW) = \nabla \cdot  \left(W\left(\nabla  h  + h \nabla \Psi \right)\right),
\quad
\frac{\partial h}{\partial \nu}  + h\frac{\partial \Psi}{\partial \nu} = 0
\end{gather}
where we define $\Psi = \phi - \Phi.$
\end{lemma}
\begin{proof}
By definition 
\begin{equation*}
\int_U gV dx = \int_U v \,dx = \mu_v, \quad \int_U h W dx = \int_U w \,dx = \mu_w.
\end{equation*}
Elementary manipulations of the definitions will show (\ref{gsolve}, \ref{hsolve}).
For simplicity assume $g,v > 0$ are smooth (the alternate case
being treated by approximation);
\begin{equation*}
\log v = \log g +  \Phi +  \log \frac{\mu_v}{\int_U e^{\Phi} \,dx} = \log g +  \Phi +  (\mathrm{const}).
\end{equation*}
Thus, 
\begin{equation*}
\frac{1}{g}\nabla g  -  \nabla  \Psi = \frac{1}{v}\nabla v - \nabla \Phi -\nabla \Psi
=\frac{1}{v}\nabla v - \nabla \phi.
\end{equation*}
This shows the second equality in \eqref{gsolve}.
Multiplying this equation by $v = gV,$ and taking the divergence 
\begin{equation*}
\nabla \cdot \left(V\left(\nabla g  -  g\nabla \Psi\right)\right)=
\frac{\partial v}{\partial t} + u \cdot \nabla v
= \frac{\partial (gV)}{\partial t} + u \cdot \nabla  (gV) 
\end{equation*}
as required. The analogous calculation proves the result for $h.$
\end{proof}

\begin{lemma}
\label{Ldiff}
There exists positive constants $C_1, C_2$ and $C_3$ depending only on
$U,$ $V,$ $\|V\|_{\inf},$ $\|V\|_{\sup},$ $W,$ $\|W\|_{\inf}$ and $\|W\|_{\sup}$ such  that
\begin{gather}
\label{Ldecay}
\frac{d}{dt} \mathscr{L}(t) \leq -C_1 \mathscr{L}(t) + C_2 \mathscr{L}(t)^2 + C_3 \mathscr{L}(t)^4.
\end{gather}
\end{lemma}
\begin{proof}
Without loss of generality, we may assume $w = W \equiv 0,$
the general case simply being a sum of the argument given below for $v,V$ 
to that of $w,W.$ 
Define $E = v - V$ and $\Psi = \phi - \Phi.$ 
Using elementary manipulations and noting that 
$E =  V(g - 1),  \nabla (E/V) = \nabla g, \partial_t E = V \partial_t g 
$ we find
\begin{equation*}
\begin{aligned}
\frac{d}{dt}\mathscr{L}(t) 
&=-\int_U  V\left(\left|\nabla \frac{E}{V}\right|^2 -  \frac{E}{V} \nabla \frac{E}{V} \cdot \nabla \Psi - 2 \nabla \frac{E}{V} \cdot \nabla \Psi + 2\left(\frac{E}{V} + 1\right) |\nabla \Psi|^2\right)\,dx  \\
& + \int_U -\Theta |\nabla u|^2 + \Theta u \cdot \nabla \phi \Delta \phi +  \frac{E}{V} u \cdot \nabla (E + V)\,dx.
\end{aligned}
\end{equation*}
By application of Young's  inequality, the last three terms in the first integral may be bounded in terms of 1/2 the first term plus products of terms involving $|E|$ and $|\nabla \Psi|$ 
of order greater than 2.  Since $\Delta \Phi \nabla \Phi$ is a gradient, \eqref{pressure}, we may insert the term $\Theta u \cdot \nabla \Phi \Delta \Phi$ in the second integral.
Then $u \cdot \nabla \Psi \Delta \Psi = u\cdot \nabla \phi \Delta \phi - u\cdot \nabla \Phi \Delta \phi - u\cdot \nabla \phi \Delta \Phi + u\cdot \nabla \Phi \Delta \Phi$ 
shows that 
\begin{equation*}
\begin{aligned}
& \int_U u \cdot \nabla \phi \Delta \phi \,dx = \int_U u \cdot \nabla \Psi \Delta \Psi + u\cdot \nabla \Phi \Delta \phi + u\cdot \nabla \phi \Delta \Phi  - u\cdot \nabla \Phi \Delta \Phi\,dx\\
&= \int_U u \cdot \nabla \Psi \Delta \Psi - u\cdot (\Delta \Phi \nabla \phi) + u\cdot \nabla \phi \Delta \Phi \,dx 
\leq \int_U a_1 | u|^2 + a_2|E|^2|\nabla \Psi|^2\,dx
\end{aligned}
\end{equation*}
where $4a_1a_2 = 1.$
We may estimate  the third term in the second integral by rewriting it as 
\begin{equation*}
\int_U -u \cdot \nabla \frac{E}{V} E + u\cdot \nabla \Phi E \,dx
\leq  \int_U b_1|u|^4 + b_2|E|^4 + (b_3 + C_{\omega}b_5)\left|\nabla  \frac{E}{V} \right|^2  + b_4\|\nabla \Phi\|_{\sup}^2|u|^2\,dx 
\end{equation*}
where $4b_1b_2b_3 = 1$ and $4b_4b_5 = 1$  and $C_{\omega}$ is the constant  in 
lemma \ref{wpi} with $V^{-1}$ in place of $\rho$ and $E$ (which has integral zero) in place of $f.$
 Choose $\Theta a_1$, $b_3$ and $b_5$  sufficiently small so 
that 
\begin{equation*}
C_4 := \min \left\{ \frac{1}{2}\|V\|_{\min} - b_3 - C_{\omega}b_5, 1 - 2\Theta a_1c_2 \right\} > 0.
\end{equation*}
With $b_1, \dots, b_5,a_1, a_2$  specified we fix $\Theta = b_4\|\nabla \Phi\|_{\sup}^2.$
Applying  lemma \ref{wpi}  and the Poincar\'e inequalities once more implies
\begin{equation*}
\begin{aligned}
\frac{d}{dt} \mathscr{L}(t) &\leq -C_4' ([\mathscr{L}(t) + \|\nabla u\|_{\mathsf{H}}^2 + \|\nabla (E/V)\|_{L^2}^2]\\
&+  C_5 \int_U |u|^4 + |E|^4 +  E^2 |\nabla \Psi|^2 + |E||\nabla \Psi|^2  \,dx
\end{aligned}
\end{equation*}
where $C_4, C_5= C_4', C_5(b_1, \dots, b_5,a_1, a_2 , c_2, C_{\omega}, C_4,  \|V\|_{\sup}, \|V\|_{\inf}).$  
 \cite{AbMeVa04} have treated the Dirichelet boundary condition case to 
arrive at an inequality similar to this one, but the Galiardo-Nirenberg-Sobolev  inequality and regularity of 
solutions to the Poisson equation are sufficient to arrive at the following;
using the estimates in 
the proof of theorem 1.3 in \cite{AbMeVa04} one may show that there are constants $C_1', C_2'$ and $C_3'$
so that 
\begin{equation*}
\begin{aligned}
&C_5 \int_U  |u|^4 + |E|^4  + E^2 |\nabla \Psi|^2 + |E||\nabla \Psi|^2  \,dx \leq \\
&C_1' \epsilon [\mathscr{L}(t) + \|\nabla u\|_{\mathsf{H}}^2 + \|\nabla (E/V)\|_{L^2}^2] + \frac{C_2'}{\epsilon} \mathscr{L}(t)^2 + \frac{C_3'}{\epsilon}L(t)^2
\end{aligned}
\end{equation*}
for all $\epsilon > 0.$  Choosing $\epsilon C_1 < C_4'$ completes the proof.
\end{proof}

\begin{lemma}
\label{subseq}
There exists a subsequence $\{t_j\}_{j=1}^{\infty}$ of $\mathbb{R}_+$ for which
\begin{equation*}
\lim_{j \rightarrow \infty} \|v(t_j) - V\|_{L^2} + 
+  \|w(t_j) - W\|_{L^2} + 
 \|\nabla \phi(t_j) - \Phi\|_{L^2} +
 \|u(t_j)\|_{L^2} = 0.
\end{equation*}
\end{lemma}
\begin{proof}
\cite{BiHeNa94} showed in the proof of theorem 6 of that article 
that lemma \ref{Kdecay} is sufficient to find a subsequence $\{t_{j'}\}_{j' = 1}^{\infty}$
of $\mathbb{R}_+$ for which the first three limits hold.  One must assume that $\sup_{\mathbb{R}_+}\|v,w\|_{L^2} < \infty$
which is the case here.  Since  $\{t_{j'}\}_{j' = 1}^{\infty}$ may be chosen from a set of positive measure,
by lemma \ref{Kdecay}, 
we may choose a second subsequence $\{t_{j''}\}_{j'' = 1}^{\infty} \subset \{t_{j'}\}_{j' = 1}^{\infty}$ for which additionally 
$\lim_{j'' \rightarrow \infty} \|\nabla u(t_{j''})\|_{L^2} = 0.$  By the Kondrakov's embedding theorem ($\|u(t)\|_{L^2}$ 
is bounded by $K(0) < \infty$ for all $t \in \mathbb{R}_+$) there is a third subsequence $\{t_j\}_{j=1}^{\infty}$ for which 
$u(t_j)$ converges to function $U \in \mathsf{V}.$   By Fatou's lemma, $\|\nabla U\|_{L^2} = 0$ so that $U$ is a constant function.
But $U \in \mathsf{V}$ then implies that $U \equiv 0,$ giving the fourth limit.
\end{proof}

Lemma \ref{Ldiff} and lemma \ref{subseq}  combined 
imply the exponential convergence to the stationary solution.
\begin{proof}[Proof of theorem \ref{l2ass}]  Let $\{t_j\}_{j=1}^{\infty}$ be the sequence
provided in lemma \ref{subseq}.  It follows from lemma \ref{Ldiff} that $\lim_{j \rightarrow \infty} L(t_j) = 0.$
Choose $j$ sufficiently large so that 
\begin{equation*}
 C_2 L(t_j)^2 
+ C_3 L(t_j)^4 \leq -\frac{C_1}{2}L(t_j).
\end{equation*}
Then for $t \geq t_j,$ $L(t)$ is nonincreasing and so 
\begin{equation*}
\frac{d}{dt}L(t) \leq -C_1 L(t) + C_2 L(t)^2 
+ C_3 L(t)^4 \leq -\frac{C_1}{2}L(t).
\end{equation*}
Setting $\lambda = C_1/2$ and $C_{\dagger} = \max\left\{1, \|V\|_{\min}^{-1}, \|W\|_{\min}^{-1}\right\}\sup_{t \in [0,t_j]} L(t)$
completes the proof.
\end{proof}

The proofs of lemmas \ref{Kdecay} through \ref{subseq} and the proof of theorem \ref{l2ass} hold
similarly  solutions of the Debye-H\"uckel system  by formally setting $u \equiv 0.$  We thus have
\begin{corollary}
\label{nonuniform}
There exists a positive constant $\lambda$ depending only on $U,$ $V,$ $\|V\|_{\inf},$ $\|V\|_{\sup},$ $W,$ $\|W\|_{\inf}$ and $\|W\|_{\sup}$
such that the solution to the Debye-Huckel system $v,w, \phi$ in dimensions $2$ and $3$ converges with rate $e^{-\lambda t}$ to steady state solution $V,W, \Phi$ in
the $L^2, L^2$ and $H^1$ norms respectively provided $\sup_{\mathbb{R}_+}\|v,w\|_{L^2} < \infty.$
\end{corollary}
\begin{remark}
This corollary in some sense improves theorem 2 of \cite{BiDo00} where the domain $U$ is assumed to be uniformly convex.
If $U$ is uniformly convex, then one may  produce a bound 
{\footnotesize
\begin{equation*}
K_{\mathrm{rel}}(t) := \int_U v \log v  + w \log w - V\log V - W \log W - \frac{1}{2}|\nabla \phi|^2 + \frac{1}{2}|\nabla \Phi|^2 \,dx 
\geq -\frac{1}{\lambda'} \frac{d}{dt} K_{\mathrm{rel}}(t)
\end{equation*}
}for some $\lambda'$ depending on $U$ by applying remark 3.7 of \cite{ArMaTo01} 
concerning a logarithmic Sobolev inequality for bounded domains and noting the convexity of $J.$  
The $L^1$ convergence of $v,w$
to $V,W$ follows then from the Csis\'ar-Kullback inequality, \cite{UnArMaTo00}.
\end{remark}
\section{Conclusion}
The equations of a incompressible, Newtonian fluid
coupled with charges in two dimensions have been studied.  The key step 
toward the existence of global in time solutions is the existence 
of a decaying entropy function which guarantees the dissipation
of kinetic and electrostatic energy and entropy.  Future avenues of study are the 
regularity of these solutions and generalizations to other incompressible Newtonian fluid systems 
coupled with polarized particles.  The state of knowlege (rather lack of knowledge) concerning the global in time existence of
weak solutions to the Debye-H\"uckel system in three dimensions prevents the generalization
of theorem \ref{globalex} in this paper to three dimensions.

\end{document}